\documentclass[11pt,reqno]{amsart}
\usepackage{amsmath}
\usepackage{amssymb}
\usepackage{amsfonts}
\usepackage{graphicx}
\usepackage{amsthm}
\usepackage{enumerate}
\usepackage{lscape}
\usepackage{dsfont}
\usepackage{color}

\newcommand{\R}{\mathds{R}}

\newcommand{\CP}{\mathds{C}\mathrm{P}}

\newcommand{\C}{\mathds{C}}

\newtheorem{theor}{Theorem}

\newtheorem{lem}[theor]{Lemma}
\newtheorem{cor}[theor]{Corollary}

\newtheorem{ex}{Example}
\newtheorem{remark}[theor]{Remark}

\begin{document}
\title[TYZ expansion of locally Hermitian symmetric spaces]{On the coefficients of  TYZ expansion of locally Hermitian symmetric spaces}
\author[A. Loi]{Andrea Loi}
\address{Dipartimento di Matematica e Informatica, Universit\`{a} di Cagliari,
Via Ospedale 72, 09124 Cagliari, Italy}
\email{loi@unica.it}
\author[M. Zedda]{Michela Zedda}
\address{Dipartimento di Matematica ``G. Peano", Universit\`{a} di Torino}
\email{michela.zedda@gmail.com}

\thanks{
The first author was supported  by Prin 2010/11 -- Variet\`a reali e complesse: geometria, topologia e analisi armonica -- Italy;
the second author was supported by the project FIRB ``Geometria Differenziale e teoria geometrica delle funzioni''. Both the authors were supported by  INdAM-GNSAGA - Gruppo Nazionale per le Strutture Algebriche, Geometriche e le loro Applicazioni.
}
\date{}
\subjclass[2000]{53C55; 58C25;  58F06}
\keywords{K\"ahler manifolds; quantization; TYZ asymptotic expansion; Hermitian symmetric spaces; scalar flat K\"ahler metrics}

\begin{abstract}
In this paper we address the problem of studying those 
K\"ahler manifolds whose first two coefficients of the associated TYZ expansion vanish and we prove that 
for a locally Hermitian symmetric space this happens only in the flat  case. We also prove that there exist nonflat 
locally Hermitian  symmetric spaces where all the odd coefficients vanish.
\end{abstract}

\maketitle

\section{introduction and statement of the main results}

Let $M$ be a $n$-dimensional complex manifold endowed with a K\"ahler metric $g$. Assume that there exists a 
holomorphic line bundle  $L$ over $M$ such that   $c_1(L)=[ \omega ]$,
where $\omega$ is the K\"{a}hler form associated to $g$ and $c_1(L)$
denotes the first Chern class of $L$
(such an $L$ exists if and only if  $\omega$ is an integral form).
Let $m\geq 1$ be a
non-negative integer and let $h_m$ be an  Hermitian metric on
$L^m=L^{\otimes m}$ such that its Ricci curvature ${\rm Ric}
(h_m)=m\omega$. Here ${\rm Ric} (h_m)$ is the two--form on $M$ whose
local expression is given by
\begin{equation}\label{rich}
{\rm Ric} (h_m)=-\frac{i}{2}
\partial\bar\partial\log h_m(\sigma (x), \sigma (x)),
\end{equation}
for a trivializing holomorphic section $\sigma :U\rightarrow
L^m\setminus\{0\}$. In the quantum mechanics terminology $L^m$ is
called the {\em quantum line bundle},
 the pair $(L^m, h_m)$ is called a {\em geometric
quantization} of the K\"{a}hler manifold $(M, m\omega)$ and $\hbar =
m^{-1}$ plays the role of Planck's constant  (see e.g.
\cite{arlquant}). Consider the separable complex Hilbert space
$\mathcal{H}_m$ consisting of global holomorphic sections  $s$ of $L^m$
such that
$$\langle s, s\rangle_m=
\int_Mh_m(s(x), s(x))\frac{\omega^n}{n!}<\infty .$$
Define:
\begin{equation}\label{Tmo}
\epsilon_{m g} (x) =\sum_{j=0}^{d_m}h_m(s_j(x), s_j(x)),
\end{equation}
where $s_j$, $j=0, \dots , d_m$ ($\dim \mathcal{H}_m=d_m+1\leq\infty$) is an orthonormal basis of $\mathcal{H}_m$.

As suggested by the notation  this function depends only on the metric $m g$ and not on   the orthonormal basis chosen or on the Hermitian metric $h_m$. Obviously if $M$ is compact $\mathcal{H}_m=H^0(L^m)$, where 
$H^0(L^m)$ is the (finite dimensional) space of global holomorphic sections of $L^{m}$.

In the literature the function $\epsilon_{m g}$ was first introduced under the name of $\eta$-{\em function} by J. Rawnsley in \cite{rawnsley}, later renamed as $\theta$-{\em function} in \cite{cgr1} followed by the {\em distortion function } of G. R. Kempf \cite{ke} and S. Ji \cite{ji}, for the special case of Abelian varieties and of S. Zhang
\cite{zha} for complex projective varieties.

In \cite{ze} Zelditch proved that if in the above setting $M$ is compact,  there exists
a complete asymptotic expansion in the $C^\infty$ category  of Kempf's distortion function:
\begin{equation}\label{TYZ}
\epsilon_{m g}(x)\sim \sum_{j=0}^\infty a_j(x)m^{n-j},
\end{equation}
where $a_0(x)=1$ and $a_j(x)$, $j=1,\dots$ are smooth functions on $M$.
This means that, for any nonnegative integers $r,k$ the following estimate holds:
\begin{equation}\label{rest}
||\epsilon_{m g}(x)-
\sum_{j=0}^{k}a_j(x)m^{n-j}||_{C^r}\leq C_{k, r}m^{n-k-1},
\end{equation}
where $C_{k, r}$ are constant depending on $k, r$ and on the
K\"{a}hler form $\omega$ and $ || \cdot ||_{C^r}$ denotes  the $C^r$
norm in local coordinates.
The expansion \eqref{TYZ} is called {\em Tian--Yau--Zelditch expansion} (TYZ in the sequel). 
Later on,  Z. Lu \cite{lu},  by means of  Tian's peak section method,
 proved  that each of the coefficients $a_j(x)$ 
is a polynomial
of the curvature and its
covariant derivatives at $x$ of the metric $g$ which can be found
 by finitely many algebraic operations.
In particular, he computed the first three coeffcients which are given  by:
\begin{equation}\label{coefflu}
\left\{\begin{array}
{l}
a_1(x)=\frac{1}{2}\rho\\
a_2(x)=\frac{1}{3}\Delta\rho
+\frac{1}{24}(|R|^2-4|{\rm Ric} |^2+3\rho ^2)\\
a_3(x)=\frac{1}{8}\Delta\Delta\rho +
\frac{1}{24}{\rm div}{\rm div} (R, {\rm Ric})-
\frac{1}{6}{\rm div}{\rm div} (\rho{\rm Ric})+\\
\qquad\quad\ \ +\frac{1}{48}\Delta (|R|^2-4|{\rm Ric} |^2+8\rho ^2)+
\frac{1}{48}\rho(\rho ^2- 4|{\rm Ric} |^2+ |R|^2)+\\
\qquad\quad\ \ +\frac{1}{24}(\sigma_3 ({\rm Ric})- {\rm Ric} (R, R)-R({\rm Ric} ,{\rm Ric})),
\end{array}\right.
\end{equation}
where $\rho$, ${\rm Ric}$, $R$, are, respectively,  the scalar curvature, the Ricci tensor and the Riemann curvature tensor
of  $(M, g)$. The reader is also referred to  
\cite{loianal} and \cite{asymsmooth} (see also formula  (\ref{recursive}) below)  for a  recursive formula  of the coefficients $a_j$'s  and  an alternative  computations of  $a_j$ for $j\leq 3$ using
Calabi's diastasis function (see also  \cite{xu1}  for a graph-theoretic
interpretation of this recursive formula).

When $M$ is noncompact, there is not a general theorem which assures the existence of an asymptotic expansion \eqref{TYZ}. Observe that in this case we say that an asymptotic expansion \eqref{TYZ} exists if \eqref{rest} holds for any compact subset $H\subset M$. M. Engli\v{s} in  \cite{englis2} showed that a TYZ asymptotic expansion exists in the case of strongly pseudoconvex bounded domain of $\C^n$ with real analytic boundary, and proved that the first three coefficients are the same as those computed by Lu for compact manifolds \eqref{coefflu}.  
The reader is also referred  to \cite{MaMarinescu} (see also \cite{commtodor} and  \cite{LoiZeddaTaub}) for the description of some curvature conditions
which assures the existence of a TYZ expansion in the noncompact case. 
Due to Donaldson's work (cfr. \cite{donaldson},
\cite{do2} and \cite{arezzoloi}) in the compact case and respectively to the theory of
quantization  in the noncompact case (see, e.g. \cite{Ber1, cgr1, cgr2, cgr3, cgr4, regcov, berezin}),
 it is  natural to study metrics with the coefficients  of  the  TYZ expansion being prescribed.
In this regards  Z. Lu and G. Tian   \cite{LuTian}  (see also \cite{engliszhang} and \cite{LoiArezzo} for the symmetric and homogenous case
respectively) prove that the PDEs $a_j=f$ ($j\geq 2$ and $f$ smooth function on $M$) are elliptic and that if the logterm
of the Bergman and  Szeg\"{o} kernel  of the unit disk bundle over $M$ vanishes then  $a_k=0$, for  $k>n$ ($n$ being the complex dimension of $M$).  
The study of these PDEs makes sense despite to the existence of  a TYZ expansion
and so given any K\"ahler manifold $(M, g)$ it makes sense to call the   $a_j$'s  the {\em coefficients associated to metric $g$}.
In the noncompact case in \cite{LoiZeddaTaub}
one can find a characterization of the flat metric as a Taub-Nut metric with $a_3=0$ while  Feng and Tu \cite{fengtu}   solve a conjecture formulated by the second authors  in \cite{zedda} by showing that  the complex hyperbolic space is the only Cartan-Hartogs domain where the coefficient  $a_2$ is constant.

In this paper we address the problem of studying those K\"ahler manifolds $(M, g)$ such that 
 the first two coeffcients   $a_1$ and $a_2$ associated to $g$   vanish (obviously this happens when the metric $g$ is flat, since in this case $a_j=0$, for all $j\geq 2$).
 By the first of  (\ref{coefflu}) a metric  with   $a_1=0$ is a   scalar flat K\"ahler  (SFK) metric.
 These metrics have been extensively analysed by several authors and we refer the reader to \cite{lebrunSFK}  for the  main results about the existence of SFK metrics in the compact case (see also  Example \ref{exlebrun} below)  and to \cite{abreuSFK} and reference therein for explicit construction of SFK in the noncompact settings.
 If  a K\"ahler  manifold is one dimensional then the condition  $a_1=0$  is clearly equivalent to flatness while by the second of (\ref{coefflu}) the condition $a_2=0$ alone yields $\Delta\rho=0$ and, so in the compact one-dimensional  case, the scalar curvature is forced to be constant. Observe also  that if the metric is not only  SFK  but also  Einstein (and so Ricci flat)
then, the   condition $a_2=0$  implies $|R|^2\equiv0$, and so, also in this case,  $g$ is flat. 
On the other hand the following example shows that there exists K\"ahler metrics
with $a_1=a_2=0$ which are not flat.
 
\begin{ex}\rm\label{exsimanca}
Let $M$ be the blown-up of $\mathds{C}^2$ at the origin and denote by $E$ the exceptional divisor.
Let $(z_1, z_2)$ be the standard coordinates of $\mathds{C}^2$. In \cite{simanca}  Simanca constructs a SFK complete (not Ricci flat) metric  $g$ on $M$, 
whose K\"ahler potential  on $M\setminus E=\mathds{C}^2\setminus\{0\}$ can be written as 
$$
\Phi(z)=|z|^2+\log |z|^2 , \ |z|^2=|z_1|^2+|z_2|^2.
$$
A straightforward computation gives:
$$|R|^2=\frac{8}{(1+|z|^2)^4}, \ |{\rm Ric}|^2=\frac{2}{(1+|z|^2)^4}.$$
Thus the coefficients $a_1$ and $a_2$ associated to $g$ both vanishes.
 \end{ex}

The following example deals with the compact case of dimension $2$.

\begin{ex}\rm\label{exlebrun}
We start by recalling that the Euler characteristic of a real $4$-dimensional  compact manifold $S$ is given by:
$$\chi (S)=\int_S [|R|^2-4|{\rm Ric|^2}+\rho^2]\  dS.$$
Therefore  it follows by (\ref{coefflu}) that  a necessary topological  condition for the 
existence of  K\"ahler  metric on a compact complex  surface with  $a_1=a_2=0$ 
is the vanishing of its Euler characteristic.
The easiest example of compact  surface which admits a  SFK 
is the product  $M=\Sigma_g\times \CP^1$ of a Riemann surface $\Sigma_g$ of genus $g\geq 2$ with a metric  of constant scalar curvature $-1$ and the one dimensional complex projective space with the metric of  constant scalar curvature $1$.
In this case the coefficient $a_2$ associated to this metric is different from zero. The Euler characteristic of $M$ is equal to $4-4g$.
LeBrun \cite{lebrunSFK} proved that if one takes the blown up  of $M$ at an even number of points (in a suitable position) one gets
a compact complex surface which still admits a SFK metric. 
Therefore by blowning  up  $M$ at $4g-4$ points one gets a compact complex surface with vanishing Euler characteristic
and which admits a SFK metric.  Unfortunately we were not able to compute $a_2$.
\end{ex}
 
By Example \ref{exsimanca} it is natural to look for  extra conditions, which  together with the vanishing  of $a_1$ e $a_2$
would imply the flatness of the  K\"ahler manifold involved. The following theorem, which represents the  first main result of this paper, shows that this is the case 
for  locally Hermitian symmetric space.

 \begin{theor}\label{main}
Let $(M, g)$ be a  locally Hermitian symmetric space of complex dimension $n$. If  the coefficients $a_1$ and $a_2$ vanish  then $g$ is flat.
\end{theor}

Throughout this paper  by a  {\em locally Hermitian symmetric space} {\em (LHSS} in the sequel)  we shall  mean a 
K\"ahler manifold $(M, g)$ whose universal cover (with the  K\"ahler metric induced by $g$) is an Hermitian symmetric space. It follows by the classification of  Hermitian symmetric spaces that the universal cover of a   LHSS  is the K\"ahler product of irreducible Hermitian symmetric spaces of compact type, irreducible  Hermitian symmetric spaces of noncompact type (namely bounded symmetric domains with a multiple of the Bergman metric) and the flat Euclidean space. Notice that if $(M, \omega)$ is a LHSS then also $(M, \lambda\omega)$ is a LHSS for each real number $\lambda >0$.

We  point out that one can  construct examples (either compact or noncompact) of  LHSS which are SFK and not flat
(take for example the product $\Sigma_g\times \CP^1$ as in Example \ref{exlebrun})
and of  LHSS with $a_2=0$ which are not SFK (see  Example \ref{exc2=0} below).
Thus, by combining these observations with Example \ref{exsimanca}, we deduce that  two of the three assumptions  ($a_1=0$,  $a_2=0$ and LHSS) in  Theorem \ref{main} are not sufficient to deduce the flatness of the metric involved.

The proof of the Theorem \ref{main} is based on the
the explicit expressions of the coefficients $a_1$ and $a_2$ associated to an  irreducible  bounded symmetric domain given by 
M. Engli\v{s} in \cite{me2} (see formula (\ref{b12}) below) which yields  $a_2-\frac12a_1^2<0$ (see Lemma \ref{coeffprop}
below).

It could be interesting to extend our result to  locally homogeneous K\"ahler spaces.
Unfortunately we were not able to attack this case due to  the lack of knowledge of the coefficient $a_1$ and $a_2$.
Nevertheless we believe our theorem holds true also in this case.

From Theorem \ref{main} it is natural to ask what happens when one imposes conditions on the other coefficients
$a_j$, with $j\geq 3$. In the following theorem which represents our second and last result, we show that there exist  nonflat LHSS where all the odd coefficients vanish.

\begin{theor}\label{main2}
Let $(\Omega,g)$ be a bounded symmetric domain,    $g=\frac1\gamma g_B$, where $g_B$ is its Bergman metric and $\gamma$ its genus, and let $(\Omega^*,g^*)$ be its compact dual. If $a_j$ (resp. $a_j^*$) are the coefficients of $(\Omega,g)$ (resp. $(\Omega^*,g^*))$ one has:
\begin{equation}\label{ajaj*}
a_j=(-1)^ja_j^*.
\end{equation}
Therefore  the odd coefficients  associated to  the metric  $g\oplus g^*$ on $\Omega\times\Omega^*$ all  vanish.
\end{theor}

\vspace{0.5cm}

The paper consists in two  more sections. In the first one we describe the Kempf distortion function for irreducible bounded domains,
we prove Lemma \ref{coeffprop} and Theorem \ref{main}.
The second one  is dedicated  to the proof of  Theorem \ref{main2}.

\section{The coefficients of  bounded symmetric domains and the proof of Theorem \ref{main}}

Let $(\Omega, g_B)$ be an irreducible bounded symmetric domain of $\mathds{C}^d$ endowed with its Bergman metric $g_B$.
Recall that $g_B$ is the metric whose associated K\"ahler form is given by $\omega_B=\frac i2 \partial \bar \partial \log K(z, z)$, where $K$ is the reproducing kernel of the Hilbert space of holomorphic function on $\Omega$ which are $L^2$ bounded with respect to the Euclidean measure of $\mathds{C}^d$.

A bounded symmetric domain $(\Omega, g_B)$ is uniquely determined by its rank $r$ and its numerical invariants $(a,b)$, $a, b\geq 0$. In particular, $\{r,a,b\}$  determine the genus $\gamma=(r-1)a+b+2$ and the dimension $d=r+\frac{r(r-1)}{2}a+rb$ of $\Omega$. In \cite[Sec. 5]{me2} M. Engli\v{s} computed the coefficients $b_j$ of $(\Omega,\frac{1}{\gamma}g_B)$ up to the third. In particular the first two read:
\begin{equation}\label{b12}
\begin{cases}
a_1=&\!\!\!\!-\frac12\gamma d,\\
a_2=&\!\!\!\!\frac18 \gamma^2d^2-\frac16\gamma^2d+\frac {1}{24}q,
\end{cases}
\end{equation}
where
$$
q=-\frac{r-1}2da^2+\frac{r-1}rad(d+r)+\frac{2d(d+r)}r.
$$
Notice that the Kempf distortion function of an irreducible bounded symmetric domain $(\Omega,\frac{1}{\gamma}g_B)$  of rank $r$ and numerical invariants $a$, $b$ is a polynomial in $m$ of degree $d$ which reads as  (see e.g. \cite[Sec. 5]{me2}) :
$$
\epsilon_{m}=\prod_{j=1}^r\frac{\Gamma(m-(j-1)\frac{a}{2})}{\Gamma(m-\frac dr -(j-1)\frac{a}{2})}.
$$

\begin{remark}\rm
Notice that in this case  the  TYZ asymptotic expansion (\ref{TYZ}) being a polynomial is finite and all the coefficients $a_j$ associated to the metric $g=\frac{1}{\gamma}g_B$ are constants.
\end{remark}
In the following lemma we prove an inequalty needed in the proof of the main theorem.
\begin{lem}\label{coeffprop}
Let $(\Omega,g=\frac{1}{\gamma}g_B)$ be an irreducible bounded symmetric domain Then the coeffcients $a_1$ and $a_2$ associated to $g$ satisfies the following inequality:
\begin{equation}\label{ineqa1a2}
a_2-\frac12a_1^2<0.
\end{equation}
\end{lem}
\begin{proof}
By  (\ref{b12}) wee need   to show that  $\frac {1}{4}q<\gamma^2d$,
namely 
$$
-\frac{r-1}2a^2+\frac{r-1}ra(d+r)+\frac{2(d+r)}r<4\gamma^2.
$$
Substituting $\gamma=(r-1)a+b+2$, $d=r(r-1)\frac{a}2+rb+r$ we get:
$$
\frac12a^2(r-1)(r-2)<4(r-1)^2a^2+4b^2+12+7(r-1)ab+13(r-1)a+14b,
$$
which holds true since $a, b\geq 0$,  $r\geq 1$ and
$\frac12(r-1)(r-2)<4(r-1)^2$.
\end{proof}

\begin{lem}\label{c2b2a2}
Let $(M_1, g_1)$ and $(M_2, g_2)$ be two  K\"ahler manifolds and consider 
their product $(M_1\times M_2, g_1\oplus g_2)$.
Denote by $a_1, a_2$, $b_1, b_2$ and $c_1, c_2$ the coefficients associated to the metrics $g_1$, $g_2$ 
and $g_1\oplus g_2$ respectively.
Then:
\begin{equation}\label{c1g1g2}
c_1=a_1+b_1,
\end{equation}
\begin{equation}\label{c2g1g2}
c_2=a_2+b_2+a_1b_1.
\end{equation}

\end{lem}
\begin{proof}
The expression for $c_1$ follows by the first of  (\ref{coefflu}) and by the fact that the scalar curvature of the
$g_1\oplus g_2$ is the sum of those of $g_1$ and $g_2$. Equation (\ref{c2g1g2}) is more subtle and can be obtained as follows.
For $j=1,2$ denote by $\Delta_j$, $R_j$, ${\rm Ric}_j$ and $\rho_j$ the Laplacian, the Riemannian tensor, the Ricci tensor and the scalar curvature of $g_j$, and by $\Delta$, $R$, ${\rm Ric}$ and $\rho$ those of $g_1\oplus g_2$. Observe that by:
$$
g_1\oplus g_2=\left(\begin{array}{c|c}
g_1&0\\
\hline
0&g_2\end{array}\right),
$$
we get:
$$
(g_1\oplus g_2)^{-1}=\left(\begin{array}{c|c}
g_1^{-1}&0\\
\hline
0&g_2^{-1}\end{array}\right),
$$
from which follows easily:
$$
\Delta \rho=\Delta_1 \rho_1+\Delta_2\rho_2,
$$
$$
|R|^2=\sum_{j,k,l,m,p,q,s,t}g^{\bar jk}g^{\bar lm}g^{\bar pq}g^{\bar st}R_{k\bar l q\bar s}\overline{R_{j\bar m p\bar t}}=|R_1|^2+|R_2|^2,
$$
$$
|{\rm Ric}|^2=\sum_{j,k,l,m}g^{\bar jk}g^{\bar lm}{\rm Ric}_{k\bar l}\overline{{\rm Ric}_{j\bar m}}=|{\rm Ric}_1|^2+|{\rm Ric}_2|^2.
$$
Thus, by the second of \eqref{coefflu} we get:
$$c_2=\Delta_1 \rho_1+\Delta_2\rho_2+\frac{1}{24}\left(|R_1|^2+|R_2|^2-4|{\rm Ric}_1|^2-4|{\rm Ric}_2|^2+3(\rho_1+\rho_2)^2\right),
$$
i.e. $c_2=a_2+b_2+a_1b_1$, as wished.
\end{proof}
\begin{ex}\rm\label{exc2=0}
Notice that there exist LHSS with vanishing $a_2$ which are not flat.
Consider for example $\left(\mathds{C}\times \mathds{C}P^1, g_0\oplus g_{FS}\right)$ and  denote by $c_1, c_2$, $a_1, a_2$, $b_1, b_2$ the coefficients associated to  $g_0\oplus g_{FS}$,  $g_0$ and to $g_{FS}$ respectively. Since $a_0=1$, $a_1=a_2=0$, $b_0=b_1=1$, $b_2=0$ by Lemma \ref{c2b2a2} one gets
$$
c_0=c_1=1,\quad c_2=0.
$$
For a compact example just take $\left(\mathds{T}\times \mathds{C}P^1,g_0\oplus g_{FS}\right)$,
where $(\mathds{T}, g_0)$ is the flat one-dimensional torus.
\end{ex}

\vspace{0.3cm}
In order to prove Theorem \ref{main} we notice that 
to any irreducible bounded symmetric domain $(\Omega,g=\frac{1}{\gamma}g_B)$ one can associate its compact dual 
$(\Omega^*, g^*)$ where $g^*$ is the pull-back of Fubini--Study metric of $\mathds{C}P^N$ via the Borel--Weil embedding $\Omega^*\rightarrow \mathds{C}P^N$. In  the affine chart $\Omega^*\setminus {\rm Cut}_p(\Omega^*)$, where ${\rm Cut}_p(\Omega^*)$ is the cut locus of the point $p\in \Omega^*$ w.r.t. $g^*$, one has:
\begin{equation}\label{duality}
\omega^*=-\frac i{2\gamma}\partial\bar \partial\log K(z,-\bar z).
\end{equation}
(where $\omega_B=\frac i2 \partial \bar \partial \log K(z, z)$).
The reader is referred to \cite{sympdual, LM01} for details and further results.

\begin{proof}[Proof of Theorem \ref{main}]
Since the coefficients $a_1$ and $a_2$ are invariants by  local isometry,  without loss of generality  we can assume that  $(M,g)$  is  simply-connected and so, by the classification theorem for Hermitian symmetric spaces  it can be written as:
$$(\Omega_1^*\times\dots\times \Omega_s^*\times \Omega_{s+1}\times\dots\times \Omega_k\times \mathds{C}^m,\lambda_1g_1^*\oplus\cdots\oplus\lambda_sg_s^*\oplus \lambda_{s+1}g_{s+1}\oplus\lambda_kg_k\oplus g_0),$$
where we denote by $g_0$ the flat metric on $\mathds{C}^m$, $(\Omega_h^*,\lambda_hg_h^*)$, $h=1,\dots, s$ are  irreducible hermitian symmetric spaces of compact type ($(\Omega_h^*,g_h^*)$ is the compact dual of $(\Omega_h,g_h)$, $h=1,\dots, s$),
$(\Omega_l, g_l)$, $l=s+1,\dots, k$, are  irreducible bounded symmetric domains and $\lambda_h$, $\lambda_l$ are  positive constants. Denote by $c_1$ and $c_2$ the coefficients associated to  $g$, by 
$a_{1,h}^*, a_{2,h}^*$ the coefficients associated to  $g_h^*$, $h=1,\dots, s$, and by $a_{1,l}, a_{2,l}$ the coefficients associated to  
$g_l$, $l=s+1,\dots, k$. By a direct computation using (\ref{coefflu}) (or  Corollary \ref{corprodotto} below) one obtains that $a_{1,h}^*=-a_{1, h}$  and $a_{2,h}^*=a_{2,h}$, where  $a_{1, h}$ and $a_{2, h}$ are the coefficients associated to the metric $g_h$, $h=1,\dots, s$.

Then, by Lemma \ref{c2b2a2}, one gets:
$$
c_1=\sum_{h=1}^s\frac{a_{1,h}^*}{\lambda_h}+\sum_{l=s+1}^k\frac{a_{1,l}}{\lambda_l}=-\sum_{h=1}^s\frac{a_{1,h}}{\lambda_h}+\sum_{l=s+1}^k\frac{a_{1,l}}{\lambda_l},
$$
$$
c_2=\sum_{h=1}^s \frac{a_{2,h}^*}{\lambda_h^2}+\sum_{l=s+1}^k \frac{a_{2,l}}{\lambda_l^2}+\sum_{h,  l}\frac{a_{1,h}^*a_{1,l}}{\lambda_h\lambda_l}=\sum_{u=1}^k \frac{a_{2,u}}{\lambda_u^2}+\frac12\left[ c_1^2-\sum_{u=1}^k\frac{a_{1,u}^2}{\lambda_u^2}\right].
$$
Since, by assumption $c_1=0$, one gets:
$$
c_2=\sum_{u=1}^k\frac{1}{\lambda_u^2} \left[a_{2,u}-\frac12a_{1,u}^2\right].
$$
By (\ref{ineqa1a2}) in  Lemma \ref{coeffprop} $a_{2,u}-\frac12a_{1,u}^2<0$ for all $u$.  Thus  $c_2=0$ 
forces $k=0$, i.e. $(M,\omega)=(\mathds{C}^m,g_0)$.
\end{proof}

\section{The proof of Theorem \ref{main2}}\label{herm}
A key ingredient in the proof of Theorem \ref{main2} is the following lemma which provide us with a  recursive formula for the 
computaton of the coefficients
$a_j$ associated to a K\"ahler manifold $(M, g)$ (we refer the reader to \cite{asymsmooth} for details).

\begin{lem}[A. Loi \cite{asymsmooth}]
Let $(M, g)$ be a  K\"ahler manifold and let $a_j(x)$, $j=0, 1,\dots$ be the coefficients associated to $g$ and denote by $a_j(x,y)$ their almost analytic extension in a neighborhood $U$ of the diagonal of $M\times M$.
Then
\begin{equation}\label{recursive}
a_k(x)=c_k+\tilde a_k(x,x)+\sum_{r+j=k, r\geq 1, j\geq 1}C_r(\tilde a_j(x,y))|_{y=x},
\end{equation}
where $a_0\equiv1$ and for all $j=1,2,\dots$:
$$
\tilde a_j(x,y)=\sum_{\alpha=0}^j a_\alpha(x,y) a_{j-\alpha}(y,x);
$$
\begin{equation}\label{cj1}
c_r(x)=C_r(1)(x);
\end{equation}
\begin{equation}\label{cjf}
C_r(f)(x)=\sum_{k=r}^{3r}\frac{1}{k!(k-r)!}L^k(f\det(g_{i\bar j})S^{k-r})|_{y=x};
\end{equation}
where if we denote by $g^{\bar jk}$ the entries of the inverse matrix of the metric $g$, $L^k$ is the operator defined by:
\begin{equation}\label{Lop}
L^k\varphi=\sum_{j_1,\dots,j_k,i_1,\dots, i_k}g^{\bar i_1 j_1}\cdots g^{\bar i_k j_k} \varphi_{j_1\cdots j_k\bar i_1\cdots \bar i_k},
\end{equation}
and
\begin{equation}\label{Sxy}
S_x(y)=- D_x(y)+\sum_{ij=1}^n g_{i\bar j}(y)(y_i-x_i)(\bar y_j-\bar x_j),
\end{equation}
where $D_x(y)$ is the diastasis function centered at $x$. 
\end{lem}
Observe that the diastasis function centered at the origin for bounded symmetric domains $(\Omega,g)$ is given by:
\begin{equation}\label{diastNC}
D_0(z,\bar z)=\frac1\gamma\log(V(\Omega)K(z,z)),
\end{equation}
where $V(\Omega)$ denotes  the total volume of $\Omega$ with respect to the Euclidean measure of the ambient complex Euclidean space (see \cite[Prop. 7]{mathann}). Further, by the discussion above, the diastasis function centered at the origin for the compact dual $(\Omega^*,g_B^*)$ reads:
\begin{equation}\label{diastC}
D_0^*(z,\bar z)=\frac1\gamma\log(V(\Omega)K(z,-z)).
\end{equation}

We are now in the position of proving Theorem \ref{main2}.
\begin{proof}[Proof of Theorem \ref{main2}]
Let $(L, h)$ be a geometric quantization of $(\Omega^*,\omega^*)$, where $\omega^*$ is the integral  K\"ahler form given by 
(\ref{duality}). Then it is not hard to see  that for all positive  integer $m$ Kempf's distortion function  $\epsilon_{m g}$ is defined 
and (by Riemann--Roch theorem)  is a  monic  polynomial in  $m$  of degree $n$, namely:
\begin{equation}\label{epsHSSCT}
\epsilon_{m g}= \sum_{j=0}^n a_j^*m^{n-j},\quad a_0^*=1,
\end{equation}
Hence, as for the case of  bounded symmetric domains,   the coefficients $a_j^*$ are constant.
We start by  proving (\ref{ajaj*}) namely  $a_j=(-1)^ja_j^*$, $j=1, 2\dots$. 

 By \eqref{recursive}, since $a_k$ is constant for all $k$, we get:
$$
a_k=a_k(0)= c_k(0)+{\tilde a}_k+\sum_{r+j=k, r\geq 1, j\geq 1} {\tilde a}_j c_r(0).
$$
 Observe that for $K=k_1+\dots+k_d$ and $J=j_1+\dots+j_d$, one has:
$$
\frac{\partial^{K+J}}{ \partial z_1^{k_1}\dots \partial z^{k_d}\partial \bar z_1^{j_1}\dots \partial \bar z_d^{j_d}} \det g|_0=0,
$$
whenever $K\neq J$, as it follows by Prop. 7 in \cite{mathann} once noticed that since $g_B$ is a K\"ahler--Einstein metric, we have $\det \left(g\right)=\frac1{\gamma^d}e^{D_0(z,\bar z)}$,
where $D_0(z,\bar z)$ is the diastasis function of $(\Omega,g_B)$ given in \eqref{diastNC}. Further it is easy to verify that $S_0(z,\bar z)$, which by definition \eqref{Sxy} reads:
$$
 S_0(z,\bar z)=- D_0(z,\bar z)+\sum_{ij=1}^d g_{i\bar j}(z,\bar z)z_i\bar z_j,
$$
where $d$ is the dimension of $\Omega$, satisfies:
$$
 S_0(0)=0,\qquad \frac{\partial^{K+J}}{ \partial z_1^{k_1}\dots \partial z^{k_d}\partial \bar z_1^{j_1}\dots \partial \bar z_d^{j_d}} S_0(z,\bar z)|_0=0,
$$
whenever $K=J$.
Thus by \eqref{cj1} and \eqref{cjf},  $c_r\neq 0$ iff $k-r=0$ and we get:
$$
c_r(0)=\frac{(-1)^r}{r!} L^r(\det( g)(y))_{|_{y=0}}.
$$
Let $D_0^*(z,\bar z)$ be the diastasis function around the origin of $g^*$ given by \eqref{diastC} and denote by $S^*$ and $c^*_r$ the operator in \eqref{Sxy} and the coefficient \eqref{cj1} associated to $g^*$.
By \eqref{duality} one gets:
$$
S_0^*(z,\bar z)=-D_0(z,-\bar z)-\sum_{ij=1}^d g_{i\bar j}(z,-\bar z)z_i\bar z_j,
$$
and thus also $S^*_0$ satisfies:
$$
 S^*_0(0)=0,\qquad \frac{\partial^{K+J}}{ \partial z_1^{k_1}\dots \partial z^{k_d}\partial \bar z_1^{j_1}\dots \partial \bar z_d^{j_d}} S^*_0(z,\bar z)|_0=0,
$$
whenever $K=J$. Further, it follows by the Einstein equation that:
$$
 \det   g^*(z,\bar z)=e^{-D_0^*(z,\bar z)},
 $$
thus, by \eqref{duality} we have:
\begin{equation}\label{dualdet}
\det \left(g\right)(z,\bar z)=\det(g^*)(z,-\bar z),
\end{equation}
and we get:
$$
\frac{\partial^{K+J}\det g^*(z,\bar z)}{ \partial z_1^{k_1}\dots \partial z^{k_d}\partial \bar z_1^{j_1}\dots \partial \bar z_d^{j_d}} |_0=(-1)^J\frac{\partial^{K+J}\det g(z,-\bar z)}{ \partial z_1^{k_1}\dots \partial z^{k_d}\partial (-\bar z_1)^{j_1}\dots \partial (-\bar z_d)^{j_d}} |_0=0,
$$
whenever $K\neq J$. Thus, also $c^*_0$ satisfies:
$$
c^*_r(0)=\frac{(-1)^r}{r!} L_*^r(\det( g)(y))_{|_{y=0}},
$$
where $L_*$ denotes the operator in \eqref{Lop} for $g^*$, and by noticing that for $K=J=r$ from \eqref{dualdet} follows:
\begin{equation}
\begin{split}
\frac{\partial^{2r}\det  g(z,\bar z)}{\partial z_1^{k_1}\dots \partial z^{k_d}\partial \bar z_1^{j_1}\dots \partial \bar z_d^{j_d}}=&(-1)^{r}\frac{\partial^{2r}\det  g^*(z,-\bar z)}{\partial z_1^{k_1}\dots \partial z^{k_d}\partial (-\bar z_1)^{j_1}\dots \partial (-\bar z_d)^{j_d}},
\end{split}\nonumber
\end{equation}
we get:
$$
c_r(0)=(-1)^r  c^*_r(0).
$$
Thus by inductive hypothesis:
$$
a_k=(-1)^k  c^*_k(0)+{\tilde a}_k+\sum_{r+j=k, r\geq 1, j\geq 1}(-1)^{r+j} {\tilde a}_j^* c^*_r(0),
$$
implies:
$$
-a_k=(-1)^k  c^*_k(0)+\sum_{\alpha, \beta}(-1)^{k}a_\alpha^* a_\beta^*+\sum_{r+j=k, r\geq 1, j\geq 1}(-1)^{r+j} {\tilde a}_j^*  c^*_r(0).
$$
Therefore formula  (\ref{ajaj*}), follows by
$$
 -a_k^*= c^*_k(0)+\sum_{\alpha, \beta}a_\alpha^* a_\beta^*+\sum_{r+j=k, r\geq 1, j\geq 1} {\tilde a}_j^*  c^*_r(0).
$$
The last part of the theorem, namely the vanishing of the odd coefficients associated to the metric $g\oplus g^*$, is a consequence 
of $a_j=(-1)^ja_j^*$ and of the following lemma and its  corollary.
\end{proof}

\begin{lem}\label{prodotto}
Let $(M_1,g_1)$, $(M_2, g_2)$ be two  K\"ahler manifolds with $\omega_1$ and $\omega_2$ integral. 
Then  the Kempf distortion function $\epsilon_{1,2}$  of $(M_1\times M_2,\omega_1\oplus \omega_2)$
is given by:
$$
\epsilon_{1,2}(x, y)=\epsilon_1(x)\epsilon_2(y),
$$
where $\epsilon_1$ (resp. $\epsilon_2$) is the Kempf distortion function associated to $(M_1,g_1)$ (resp.  $(M_2,g_2))$.
\end{lem}
\begin{proof}
For $\alpha=1,2$ let $(L_\alpha, h_\alpha)$ be the Hermitian line bundle over $M_\alpha$ such that ${\rm Ric} (h_\alpha)=\omega_\alpha$ (cfr. (\ref{rich}) in the introduction) and let $(L_{1, 2}, h_{1, 2})$ be the  holomorphic hermitian  line bundle over $M_1\times M_2$ such that ${\rm Ric} (h_{1, 2})=\omega_1\oplus \omega_2$. 
Let 
$$
\mathcal{H}_\alpha=\left\{s\in H^0(L_\alpha)\  |\ \int_{M_\alpha} h_\alpha(s,s)\frac{\omega_\alpha^{n_\alpha}}{n_\alpha!}<\infty\right\},
$$
where $n_\alpha$ is the complex dimension of $M_\alpha$ (if $M_\alpha$ is compact then $\mathcal{H}_\alpha\equiv  H^0(L_\alpha)$). Let $\{s^1_j\}$ (resp. $\{s^2_k\}$) be a orthonormal basis for $\mathcal{H}_1$ (resp. $\mathcal{H}_2$) with respect to the $L^2$-product
induced by $h_1$ (resp. $h_2$).
Fix a  local trivialization  $\sigma_{\alpha}:U_\alpha\rightarrow L_\alpha$ of $L_\alpha$ ($\alpha=1, 2$)
on an open and dense subset $U_\alpha\subset M_\alpha$ and 
let $\sigma :U_1\times U_2\rightarrow L_{1, 2}$ be the local trivialization of $L_{1,2}$
given by $\sigma_{1, 2} (x, y)=\sigma_1 (x)\sigma_2 (y)$.
Let $s_{j, k}$ be the global  holomorphic sections of $L_{1, 2}$ 
which in the trivialization $\sigma_{1, 2}$ are given by: 
$s_{j, k}(x, y)=f_j^1(x)f_k^2(y)\sigma_{1, 2}(x, y)$,
where $s_j^1(x)=f_j^1(x)\sigma_1(x)$
and $s_k^2(y)=f_k^2(y)\sigma_2(y)$.
 
We claim that  $s_{j, k}$ is an orthonormal basis for the Hilbert  space:
$$
\mathcal{H}_{1, 2}=\left\{s\in H^0(L_{1,2}) \  |\ \int_{M_1\times M_2} h_{1, 2}(s,s)\frac{(\omega_1\oplus\omega_2)^{n_1+n_2}}{(n_1+n_2)!}<\infty\right\}$$
$$=\left\{s\in H^0(L_{1,2}) \  |\ \int_{M_1\times M_2} h_{1, 2}(s,s)\frac{\omega_1^{n_1}\wedge\omega_2^{n_2}}{n_1!n_2!}<\infty\right\}.
$$
Indeed, in the the trivialization $\sigma_\alpha$ and $\sigma_{1, 2}$ the Hermitian product $h_{\alpha}$
and $h_{1, 2}$ are given respectively by $h_\alpha (\sigma_\alpha (x),\sigma_\alpha (x) ) =e^{-\Phi_\alpha (x)}$
and $h_{1, 2}(x, y)=e^{-(\Phi_1 (x)+\Phi_2 (y))}$,
where $\Phi_\alpha :U_\alpha\rightarrow \R$, $\alpha =1, 2$, is a   K\"ahler potential
for $\omega_{\alpha}$ ($\omega_\alpha =\frac{i}{2}\partial\bar\partial\log \Phi_\alpha$).
Thus 
$$
\int_{M_1\times M_2}h_{1, 2}(s_{j, k}, s_{l, m} )\frac{\omega_1^{n_1}\wedge\omega_2^{n_2}}{n_1!n_2!}=\int_{M_1\times M_2}e^{-(\Phi_1 +\Phi_2)}f_j^1f^2_k\overline{f_l^1f^2_m}\frac{\omega_1^{n_1}\wedge\omega_2^{n_2}}{n_1!n_2!}$$
$$
=\int_{M_1}h_1(s_j^1, s_l^1)\frac{\omega_1^{n_1}}{n_1!}\int_{M_2}h_2(s_k^2, s_m^2)\frac{\omega_2^{n_2}}{n_2!}=\delta_{jl}\delta_{km},
$$
and the claim is proved.
Hence
$$\epsilon_{1, 2}=\sum_{j, k}h_{1, 2}(s_{j, k}, s_{j, k})=\sum_jh_1(s_j^1, s_j^1)\sum_kh_2(s_k^2, s_k^2)=\epsilon_1\epsilon_2.$$
\end{proof}

\begin{cor}\label{corprodotto}
Let $(M_1,g_1)$, $(M_2, g_2)$ be as in the  lemma. Assume that for $m$ sufficiently large 
there exist   the TYZ expansions 
$$\epsilon_{1, m}(x)\sim \sum_{j=0}^\infty a_j(x)m^{n-j}, \quad   \epsilon_{2, m}(y)\sim \sum_{j=0}^\infty b_j(y)m^{n-j},$$ 
$$\epsilon_{1, 2, m}(x, y)\sim \sum_{j=0}^\infty c_j(x, y)m^{n-j},$$
 of the Kempf distortion functions $\epsilon_{1, m}(x)$, $\epsilon_{2, m}(y)$ and $\epsilon_{1, 2, m}(x, y)$
of $(M_1, g_1)$, $(M_2, g_2)$ and $(M_1\times M_2, g_1\oplus g_2)$ respectively.
Then  
$$
c_j(x, y)=\sum_{\alpha+\beta=j}a_\alpha (x) b_\beta (y).
$$
\end{cor}

\begin{remark}\label{trugen}\rm
We believe that the conclusion of the  previous corollary holds true for the product of two K\"ahler manifolds  without any assumption  on the existence of a TYZ expansion.
\end{remark}

\begin{remark}\rm
It is worth pointing out that Kempf's  function of a LHSS does not always exist. For instance, let $(\CP^1\times \CP^1, \sqrt{2}g_{FS}\oplus g_{FS})$. Then the corresponding  K\"ahler form 
$\lambda(\sqrt2\omega_{FS}\oplus \omega_{FS})$  is not integral for any $\lambda\in\mathds{R}$. The integrality of the K\"ahler form is a necessary condition for Kempf's distortion function to exist, although it is not sufficient, as one can see for example when $(\Omega,\beta g_B)$ is a bounded symmetric domain endowed with its Bergman metric $g_B$, since $\beta g_B$ is balanced  iff $\beta>\frac{\gamma-1}{\gamma}$ (see \cite[Th. 1]{balancedCH}).

Notice also that when the Kempf's distortion function does exist, the rest in \eqref{rest} is zero (i.e. the Kempf's distortion function is a polynomial in $m$) if and only if the LHSS $M$ is simply connected. Indeed, when Kempf's distortion function is a polynomial in $m$, since its coefficients are constant it also is and in particular there exists an isometric and holomorphic immersion $f\!: M\rightarrow\mathds{C}P^N$ for some $N\leq \infty$ (see \cite{balancedCH} for details). Consider now the universal covering $\pi\!: \tilde M\rightarrow M$. Since $\tilde M$ is a Hermitian symmetric space, there exists an {\em injective}, isometric and holomorphic immersion $F\!:\tilde M\rightarrow\mathds{C}P^N$ into $\mathds{C}P^N$ for some $N\leq \infty$ (see \cite[Lemma 2.1]{loi06} for the injectivity of  $F$). On the other hand, the composition $f\circ \pi$, is also an isometric and holomorphic K\"ahler immersion of $\tilde M$ into $\mathds{C}P^N$, which is {\em not injective} unless $M$ is simply connected. A contradiction then  follows from Calabi rigidity's Theorem \cite{calabi}. Viceversa, if $M$ is simply connected and noncompact, then it is a polynomial (see e.g. \cite[Ex. 2.14 p. 431]{englisBerezin} and references therein), while if it is simply connected and compact, then Kempf's distortion is constant and by Riemann--Roch theorem it is a polynomial.
\end{remark}


\begin{thebibliography}{99}

\bibitem{abreuSFK} M. Abreu,
{\em Scalar-flat K\"{a}hler metrics on non-compact symplectic toric $4$-manifolds},
Ann. of Glob. Anal. and Geom., Vol. 41 (2012), 209-239.



\bibitem{arlquant} C. Arezzo, A. Loi,
{\em Quantization of K\"{a}hler
manifolds and the asymptotic
expansion of Tian--Yau--Zelditch},
J. Geom. Phys. 47  (2003), 87-99.

\bibitem{arezzoloi}
C. Arezzo, A. Loi,
{\em Moment maps, scalar curvature and quantization of K\"ahler manifolds},
Comm. Math.  Phys. 243 (2004), 543-559.


\bibitem{LoiArezzo} C. Arezzo, A. Loi, F. Zuddas, \emph{Szeg\"o kernel, regular quantizations and spherical CR-structures}, 
Math. Z.  (2013) 275, 1207-1216.




\bibitem{cgr1} M. Cahen, S. Gutt, J. H. Rawnsley,
{\em Quantization of K\"{a}hler manifolds I: Geometric
interpretation of Berezin's quantization}, JGP. 7 (1990), 45-62.

%
%

\bibitem{calabi} E. Calabi,
{\em Isometric Imbeddings of Complex Manifolds},
Ann. Math. 58 (1953), 1-23.





\bibitem{donaldson} S. Donaldson, {\em Scalar Curvature and Projective Embeddings, I}, J. Diff. Geometry 59 (2001),  479-522.

\bibitem{do2} S. Donaldson,
{\em Scalar Curvature and Projective Embeddings, II},
Q. J. Math. 56 (2005),  345--356.

\bibitem{englisBerezin} M. Engli\v{s}, {\em Berezin quantization and reproducing kernels on complex domains},
Trans. Amer. Math. Soc. 348 (1996), no. 2, 411--479. 

\bibitem{englis2} M. Engli\v{s}, {\em A Forelli-Rudin construction and asymptotics of weighted Bergman kernels}, J. Funct. Anal. 177 (2000), no. 2, 257--281.

\bibitem{me2} M. Engli\v{s},
{\em The asymptotics of a Laplace integral on a K\"{a}hler manifold},
J. Reine Angew. Math. 528 (2000), 1-39.


\bibitem{engliszhang} M. Engli\v{s}, G. Zhang, {\em Ramadanov conjecture and line bundles over compact Hermitian symmetric spaces},  Math. Z., vol. 264, no. 4 (2010), 901-912.

\bibitem{fengtu} Z. Feng, Z. Tu, {\em On canonical metrics on Cartan-Hartogs domains}, Math. Zeit. 278 (2014), Issue 1-2, 301--320.



\bibitem{ke} G. R. Kempf,
{\em Metric on invertible sheaves on abelian varieties},
Topics in algebraic geometry (Guanajuato)  (1989).



\bibitem{ji} S. Ji,
{\em Inequality for distortion function of invertible
sheaves on Abelian varieties},
Duke Math. J. 58 (1989),  657-667.


\bibitem{lebrunSFK} C. LeBrun, 
{\em Scalar flat K\"ahler metrics on blown-up ruled surfaces },
J. reine angew. Math. 420 (1991), 161-177.

\bibitem{loianal} A. Loi,
{\em The Tian--Yau--Zelditch asymptotic expansion
for real analytic K\"{a}hler metrics},
Int. J.  of Geom. Methods
Mod. Phys. 1 (2004), 253-263.



\bibitem{asymsmooth} A. Loi, A Laplace integral, the T-Y-Z expansion and Berezin's transform on a K\"ahler manifold, {\em Int. J. Geom. Meth. in Mod. Ph.} {\bf 2} (2005), 359--371. 


\bibitem{loi06} A. Loi, {\em Calabi's diastasis function for Hermitian symmetric spaces}, Diff. Geom. Appl. 24 (2006), 311-319.

\bibitem{commtodor} A. Loi, T. Gramchev,
{\em TYZ expansion for the Kepler manifold},
Comm. Math.  Phys. 289, (2009), 825-840.



\bibitem{mathann} A. Loi, M. Zedda, \emph{K\"ahler--Einstein  submanifolds of   the infinite dimensional  projective space}, Math. Ann. 350 (2011), no. 1, 145--154. .


\bibitem{balancedCH} A. Loi, M. Zedda, \emph{Balanced metrics on Cartan and Cartan--Hartogs domains}, 
 Math. Z. 270 (2012), no. 3-4, 1077-1087.

\bibitem{LoiZeddaTaub}A. Loi, M. Zedda, F. Zuddas \emph {Some remarks on the K\"ahler geometry of the Taub-NUT metrics }, Ann. of Glob. Anal. and Geom., Vol. 41 n.4 (2012), 515--533.

\bibitem{lu} Z. Lu,
{\em On the lower terms of the asymptotic expansion of Tian--Yau--Zelditch},
Amer. J. Math. 122 (2000),  235-273.

\bibitem{LuTian} Z. Lu, G. Tian, \emph{The log term of Szeg\"o Kernel}, Duke Math. J. 125, N 2 (2004), 351-387.

\bibitem{MaMarinescu} X. Ma, G. Marinescu, \emph{Holomorphic morse inequalities and Bergman kernels}, Progress in Mathematics, Birkh\"auser, Basel, (2007).


\bibitem{rawnsley} J. Rawnsley, \emph{Coherent states and K\"ahler manifolds}, Quart. J. Math. Oxford (2), n. 28 (1977), 403--415.

\bibitem{simanca} S. R. Simanca,
{\em K\"ahler metrics of constant scalar curvature on bundles over
${\C}P_{n-1}$}, Math. Ann. 291, 239-246 (1991).

\bibitem{xu1} H. Xu,
{\em A closed formula for the asymptotic expansion of the Bergman kernel},
Comm. Math. Phys. 314 (2012), no. 3, 555Ð585. 



\bibitem{zedda} M. Zedda, {\em Canonical metrics on Cartan-Hartogs domains},
Int. J. Geom. Methods Mod. Phys. 9 (2012), no. 1, 1250011, 13 pp.


\bibitem{ze} S. Zelditch,
{\em Szeg\"{o} Kernels and a Theorem of Tian},
Internat. Math. Res. Notices  6 (1998), 317--331.
\bibitem{zha} S. Zhang,
{\em Heights and reductions of semi-stable varieties},
Comp. Math. 104 (1996), 77-105.

\end{thebibliography}
\end{document}